\documentclass{amsart}

\newtheorem{theorem}{Theorem}[section]
\newtheorem{lemma}[theorem]{Lemma}
\newtheorem{proposition}[theorem]{Proposition}
\newtheorem{corollary}[theorem]{Corollary}

\theoremstyle{definition}

\theoremstyle{remark}
\newtheorem{remark}[theorem]{Remark}

\DeclareMathOperator{\Ricci}{Ric}

\numberwithin{equation}{section}

\usepackage{graphics}

\begin{document}

\title{Critical Tori for Mean Curvature Energies in Killing Submersions}


\author{\'Alvaro P\'ampano}
\address{Department of Mathematics, University of the Basque Country, Bilbao, Spain.}
\curraddr{Department of Mathematics, Idaho State University, Pocatello, ID, U.S.A.}
\email{alvaro.pampano@ehu.eus}
\thanks{The author would like to thank Professor J. Arroyo for the images of Figure \ref{dib} and for kindly reviewing the paper.}



\date{\today}

\dedicatory{}

\begin{abstract}
We study surface energies depending on the mean curvature in total spaces of Killing submersions, which extend the classical notion of Willmore energy. Based on a symmetry reduction procedure, we construct vertical tori critical for these mean curvature energies. These vertical tori are based on closed curves critical for curvature energy functionals in Riemannian 2-space forms.

The binormal evolution of these critical curves in Riemannian 3-space forms generates rotational tori solutions for an ample family of Weingarten surfaces. Therefore, we also introduce some correspondence results between these two types of tori and illustrate their relation.
\end{abstract}

\maketitle

\section{Introduction}\label{Sec1}

For an immersed surface, the \emph{mean curvature} is, arguably, the most important extrinsic invariant. Indeed, in many cases, it gives us enough information to understand the extrinsic geometry of an immersed surface into the ambient space. The term mean curvature was first coined by S. Germain, although it was previously used by J. B. Meusnier to characterize surfaces that locally minimize area.

In 1811, S. Germain suggested that the free energy controlling the physical system associated with an elastic plate should be measured by an integral over the plate surface. She also pointed out that the integrand should be a symmetric and even-degree polynomial in the principle curvatures, \cite{G}.

One of the simplest functionals of this type is the so-called \emph{total squared curvature} or \emph{bending energy}. For a surface $S$ in $\mathbb{R}^3$ the bending energy is given by
$$\mathcal{W}(S)=\int_S H^2\,dA\,,$$
where $H$ denotes the mean curvature of $S$.

Later, in the 1920's, Blaschke's school studied the variational problem associated with the bending energy for surfaces immersed in $\mathbb{R}^3$, \cite{Bl}. In particular, G. Thomsen obtained the Euler-Lagrange equation of $\mathcal{W}(S)$ by means of computing the first variation formula, \cite{T} .

The bending energy was reintroduced years later by T. J. Willmore who investigated the minima for the bending energy within a given topological class, obtaining many interesting results, \cite{Wi}. Therefore, the functional $\mathcal{W}(S)$ is also known as the \emph{Willmore energy}. As a particular case, in 1965, he proposed the following conjecture: \emph{the Willmore energy of every smooth immersed torus is greater or equal to $2\pi^2$, with equality for the Clifford torus}.

Since then, there has been an intensive investigation about this subject (see, \cite{B}, \cite{BFG}, \cite{N}, \cite{R}, \cite{To},...). In fact, the conjecture was proved by F. C. Marques and A. Neves in \cite{MN}. The notion of Willmore surfaces has also been generalized to other ambient spaces so that the conformal invariance property is preserved, \cite{Ch}, although there are not many results in background spaces with non-constant sectional curvature.

In ambient spaces of constant curvature, this generalization is done by introducing the constant curvature in the integrand. Furthermore, in this setting, the Willmore energy and its critical points, \emph{Willmore surfaces}, have strong connections in other areas with applications to the analysis of elastic plates, to bosonic string theories, computer vision and sigma models \cite{BCO}, to mention a few. Several of these applications are based on a beautiful link between Willmore surfaces and elastica which can be established using a symmetry reduction procedure \cite{Pi}.

The theory of elastic curves is a classical subject which has been studied by Galileo, the Bernoulli family, Euler, Kirchhoff, Born and many others. In particular, elasticae played an important role in the development of Calculus of Variations. For a detailed historical background, we refer the reader to \cite{L}.

Following a classical model of D. Bernoulli, \emph{elastic curves} are the minimizers (or, more generally, critical curves) of the bending energy. Notice that, for curves $\gamma$ isometrically immersed in a Riemannian manifold, $H$ is nothing but the geodesic curvature of $\gamma$, $\kappa$, and $\mathcal{W}(S)$ reduces to
$$\mathcal{E}(\gamma)=\int_\gamma \kappa^2\,ds\,.$$
Using this formulation of elastica as a variational problem, L. Euler described the possible qualitative types of elastic curves in $\mathbb{R}^2$ with a constraint on the length of the curves, \cite{E}.

In recent decades, the notion of elastic curves has been generalized in different ways. On one hand, different ambient spaces have been considered \cite{BG} and; on the other hand, the integrand has been substituted by arbitrary functions depending on the curvatures of the curve \cite{ACG}.

Returning to surfaces, in contrast to Willmore energies, there are motivations in Physics and Biology that have lead researchers to consider more complicated curvature functionals. For instance, in 1973, W. Helfrich proposed a model for fluid membranes given by the functional \cite{H}
$$\mathcal{H}(S)=\int_S\left(a\left(H+b\right)^2+c K_S\right)\,dA$$
where $a$, $b$ and $c$ are some constants depending on the material of the membrane and $K_S$ is the \emph{Gaussian curvature} of the surface. Recently, extensions to arbitrary functionals depending on the mean and Gaussian curvatures have been studied in the literature for surfaces immersed in 3-space forms (see \cite{GTT} and references therein).

However, given the complexity of modern physical and biophysical models, beyond ambient spaces with constant curvature, it seems natural to consider also surface immersions into ambient spaces with fewer symmetries. 

In this paper, we study critical tori for energies depending on their mean curvature (here, these energies are going to be called \emph{mean curvature energies}) in total spaces of Killing submersions. A remarkable family includes the homogeneous 3-spaces. A Riemannian 3-manifold is said to be homogeneous if for every two points there exists an isometry mapping them. For homogeneous 3-spaces there are three possibilities for the degree of rigidity, since they may have an isometry group of dimension $6$, $4$ or $3$. The maximum rigidity $6$, corresponds to 3-space forms. Thus, this characteristic makes homogeneous 3-spaces (and, therefore, Killing submersions) ideal for modeling physical systems, which have stimulated the interest of many authors.

In Section \ref{Sec2}, we define the mean curvature energy functional for any arbitrary function acting on the space of isometric immersions of a surface in a 3-dimensional Riemannian manifold and we obtain the associated Euler-Lagrange equation.

Then, in Section \ref{Sec3}, focusing mainly on surfaces living in total spaces of Killing submersions, we establish a connection between critical tori for mean curvature energies and closed critical curves for curvature energies in the base surface.

Several families of examples are shown in Section \ref{Sec4} by restricting ourselves to homogeneous 3-spaces. In this setting, total spaces of Killing submersions can be locally described as Bianchi-Cartan-Vranceanu spaces while the base surface happens to have constant curvature. Therefore, since the variational problem for curvature energies acting on the space of isometrically immersed curves in surfaces with constant curvature is well known for some cases, we can construct many examples of critical tori in the corresponding homogeneous 3-space.

It turns out that after binormal evolution, critical curves in these surfaces are the profile curves of rotational tori in 3-space forms. Section \ref{Sec5} is devoted to the construction of this new family of tori and to proving that they are Weingarten surfaces, i.e. they verify a functional relation between their principal curvatures.

In Section \ref{Sec6}, we describe this correspondence between critical tori for mean curvature energies in total spaces of Killing submersions and Weingarten rotational tori in 3-space forms. We finish the paper with an illustration of an interesting particular case related to the standard Hopf mapping between spheres. This case connects a Blaschke's type mean curvature energy with minimal tori in 3-spheres via the Blaschke's variational problem in the 2-sphere.

\section{Mean Curvature Energies}\label{Sec2}

Let $(M,g=\langle,\rangle)$ be a 3-dimensional Riemannian manifold, with Levi-Civita connection $\bar{\nabla}$, and $S$ be a surface. For any isometric immersion of $S$ in $M$, $\phi:S\rightarrow M$ we denote by $dA_\phi$ the induced area element obtained via $\phi$. If we fix $\eta$ to be a unit normal vector field of $\phi$, we have that the \emph{mean curvature vector} of $\phi$ can be written as $${\vec H}_\phi=H_\phi \eta\,,$$ where $H_\phi$ denotes the \emph{mean curvature function} of $\phi$. 

Then, in the space of isometric immersions of $S$ in $M$, denoted throughout the paper by $I(S,M)$, we define the following \emph{mean curvature energy} functional
\begin{equation}\label{energy}
\mathcal{F}(S)\equiv \mathcal{F}(S,\phi):=\int_S P\left(2 H_\phi\right)\,dA_\phi\,,
\end{equation}
where $P(u)$ is a smooth function in an adequate domain.

A \emph{variation} of $\phi\in I(S,M)$ is a smooth map $\Phi: S\times \left(-\varepsilon,\varepsilon\right)\rightarrow M$ where $\Phi(p,0)\equiv \phi_0(p)=\phi(p)$, $\forall\, p\in S$ and which satisfies that for any $\varsigma\in\left(-\varepsilon,\varepsilon\right)$, the map $\phi_\varsigma\equiv \Phi(-,\varsigma)$ belongs to $I(S,M)$. In this setting, there exists a vector field along $\phi$,
$$V(p)\equiv V(p,0):=\Phi_*\left(\frac{\partial}{\partial\varsigma}\left(p,\varsigma\right)\right)\big\lvert_{\varsigma=0}\,,$$
which is called the \emph{variation vector field} associated to $\Phi$. Thus, we can identify the tangent space $T_\phi\left(I(S,M)\right)$ with the space of vector fields along $\phi$ and, consequently, we have
$$\partial\mathcal{F}(S,\phi)[V]=\frac{\partial}{\partial\varsigma}_{\big\lvert_{\varsigma=0}}\left(\int_S P\left(2 H_{\phi_\varsigma}\right)\,dA_{\phi_\varsigma}\right).$$
Notice that for closed surfaces $S$, after reparametrization if necessary, the variations may be assumed to be normal to $S$. Hence, the one-parameter family of immersions $\Phi(p,\varsigma)=\phi_\varsigma(p)$ can be given by
$$\phi_\varsigma(p)=\phi_0(p)+\varsigma \varphi(p)\eta(p)\,,$$
where $\varphi$ is a smooth function. Moreover, in this case the variation vector field simplifies to $V(p)=\varphi(p)\eta(p)$.

The main purpose of this section consists of computing the field equations, or Euler-Lagrange equations, associated with the functionals $\mathcal{F}(S)$, \eqref{energy}. To this end, we collect some formulae that will be needed in the sequel in the following lemmas. For the sake of simplicity, proofs are omitted. They can be obtained using similar computations to those included in \cite{W}. For the same reason, the symbol $\phi$ is avoided in our notation.

We begin by introducing the variation of the mean curvature vector field of $\phi$.

\begin{lemma}\label{lem1} Let $\vec{H}(p,\varsigma)$ be the mean curvature vector field of $\phi_\varsigma$ at $p\in S$. Then
$$D_\varsigma {\vec H}_{\big\lvert_{\varsigma=0}}=\frac{1}{2}\left(\Delta V^{\perp}+\widetilde{A}(V^\perp)+\Ricci(\eta,\eta)V^\perp\right)+D_{V^T}{\vec H}\,,$$
where $\Delta$ is the Laplacian associated with the connection $D$ in the normal bundle of $\phi(S)$, $\widetilde{A}$ stands for the Simons operator \cite{W} and $\Ricci$ is the Ricci curvature of $M$. Here, $(\,)^T$ and $(\,)^\perp$ denote tangential and normal components, respectively.
\end{lemma}

In the following lemma, the expression for the variation of the area element of $\phi$ is described.

\begin{lemma}\label{lem2} Let $dA_{\phi_\varsigma}$ be the induced area element via the immersion $\phi_\varsigma$. Then its variation is given by
$$\frac{\partial}{\partial\varsigma}_{\big\lvert_{\varsigma=0}} dA_{\phi_\varsigma}=-2\langle {\vec H},V\rangle dA_\phi+d\Theta\,,$$
where $\Theta$ is the one-form defined by $\Theta(X)=d A_\phi\left(V^T,X\right)$.
\end{lemma}

Now, we are in conditions to prove the main result of the section.

\begin{theorem} Let $I(S,M)$ be the space of isometric immersions of a surface $S$ in a 3-dimensional Riemannian manifold $M$. Consider the mean curvature energy functional
$$\mathcal{F}(S)=\int_S P\left(2 H_\phi\right)\,dA_\phi$$
acting on $I(S,M)$. Then the Euler-Lagrange equation of $\mathcal{F}(S)$ is
\begin{equation}
\Delta P' +P'\left(4 H_\phi^2-2K_S+2R+\Ricci (\eta,\eta)\right)-4 P H_\phi=0\,, \label{EL}
\end{equation}
where $P'$ represents the derivative of $P$ with respect to $H_\phi$ and $R$ is the extrinsic Gaussian curvature.
\end{theorem}
\begin{proof} We begin by computing the first variation formula for $\mathcal{F}(S)$, \eqref{energy}. We have (we are omitting the symbol $\varsigma$, for simplicity)
\begin{eqnarray}
\frac{\partial}{\partial \varsigma}\mathcal{F}(S)&=&\frac{\partial}{\partial \varsigma}\left(\int_S P(2H_\phi)dA_\phi\right)=\int_S \frac{\partial}{\partial \varsigma}\left(P(2H_\phi)\right)dA_\phi+P(2H_\phi)\frac{\partial}{\partial \varsigma}\left(d A_\phi\right)\nonumber\\
&=&\int_S P' \frac{\partial}{\partial \varsigma}\left(H_\phi\right)dA_\phi+P\frac{\partial}{\partial \varsigma}\left(dA_\phi\right)\,,\label{FVF}
\end{eqnarray}
where $P'$ is the derivative of $P$ with respect to $H_\phi$. Now, using Lemma \ref{lem1}, we have that
\begin{eqnarray*}
\frac{\partial}{\partial \varsigma}\left(H_\phi\right)&=&\frac{\partial}{\partial \varsigma}\left(\langle {\vec H}_\phi,\eta\rangle\right)=\langle \frac{\partial}{\partial \varsigma} \left( {\vec H}_\phi\right),\eta\rangle\\
&=& \langle \frac{1}{2}\left(\Delta V^{\perp}+\widetilde{A}(V^\perp)+\Ricci(\eta,\eta)V^\perp\right)+D_{V^T}{\vec H}_\phi,\eta\rangle\,.
\end{eqnarray*}
Thus, substituting this in \eqref{FVF} and applying Lemma \ref{lem2} for the second term in \eqref{FVF}, 
\begin{eqnarray*}
\frac{\partial}{\partial \varsigma}_{\big\rvert_{\varsigma=0}}\mathcal{F}(S)&=&\int_S \frac{P'}{2}\langle \Delta V^{\perp}+\widetilde{A}(V^{\perp})+\Ricci(\eta,\eta)V^\perp,\eta\rangle dA_\phi-\int_S 2P \langle {\vec H}_\phi,V\rangle dA_\phi\\
&&+\int_S P'\langle D_{V^T}{\vec H}_\phi,\eta\rangle dA_\phi+\int_S P d\Theta\,.
\end{eqnarray*}
Notice that the second line in the above formula gives, after integration, boundary terms. From now on, since we are only interested in obtaining the associated Euler-Lagrange equation, it is enough to consider normal variations, i.e. $V=\varphi\eta$. As a consequence, we can omit the boundary terms. Then, by standard arguments involving integration by parts and properties of $\widetilde{A}$, we obtain the following Euler-Lagrange equation
$$\frac{1}{2}\Delta P'+\frac{P'}{2}\langle\widetilde{A}(\eta),\eta\rangle+\frac{P'}{2}\Ricci(\eta,\eta)-2P{H}_\phi=0\,.$$
Finally, from the definition of $\widetilde{A}$, we have that $\widetilde{A}(\eta)=\left(4H_\phi^2-2K_S+2R\right)\eta$, and then the Euler-Lagrange equation boils down to
$$\Delta P' +P'\left(4 H_\phi^2-2K_S+2R+\Ricci (\eta,\eta)\right)-4 P H_\phi=0\,,$$
proving the statement. \end{proof}

\begin{remark} Assume that $M$ is a Riemannian 3-space form with constant sectional curvature $\rho$. Then, $2R+\Ricci (\eta,\eta)=4\rho$ and the Euler-Lagrange equation of $\mathcal{F}(S)$, \eqref{energy}, reduces to 
$$\Delta P' +P'\left(4 H_\phi^2-2K_S+4\rho\right)-4 P H_\phi=0\,.$$
This equation coincides with the Euler-Lagrange equation obtained in \cite{GTT} when the energy does not depend on the Gaussian curvature.\\
In particular, let $P(2H_\phi)=H_\phi^2$ and $\rho=0$, i.e. $\mathcal{F}(S)$ is the (classical) Willmore energy $\mathcal{W}(S)$ in $\mathbb{R}^3$. Then $P'(2H_\phi)=2H_\phi$, since $P'$ denotes the derivative of $P$ with respect to $H_\phi$ (do not confuse with the derivative with respect to the parameter of $P$). We substitute this in the above formula, obtaining
$$\Delta H_\phi+2H_\phi\left(H_\phi^2-K_S\right)=0\,,$$
as expected.
\end{remark}

Finally, observe that closed (compact with no boundary) surfaces solutions of the Euler-Lagrange equation \eqref{EL} are critical points (or, simply, extremals) of $\mathcal{F}(S)$, \eqref{energy}, since the boundary term appearing in the First Variation Formula vanishes.

\section{Critical Vertical Tori}\label{Sec3}

In this section we study critical tori of the mean curvature energies $\mathcal{F}(S)$, \eqref{energy}. Our analysis focuses mainly on tori living in the total space of a Killing submersion.

Let $M$ be a 3-dimensional Riemannian manifold as in the previous section and consider a fixed surface $B$. A Riemannian submersion $\pi:M\rightarrow B$ of $M$ over $B$ is called a \emph{Killing submersion} if its fibers are the trajectories of a complete unit Killing vector field, $\xi$. Fibers of Killing submersions are geodesics in $M$ and form a foliation called the \emph{vertical foliation}. Indeed, the Killing vector field $\xi$ is sometimes referred to as the \emph{vertical Killing} vector field.

Since $\xi$ is a vertical unit Killing vector field, then it is clear that for any vector field $Z$ on $M$, there exists a function $\tau_Z$ such that $$\bar{\nabla}_Z\xi=\tau_Z Z\wedge \xi\,,$$ where $\wedge$ denotes the usual vector product in 3-dimensional manifolds. Indeed, as proved in \cite{M}, one can see that $\tau_Z$ does not depend on the vector field $Z$, so we have a function $\tau\in\mathcal{C}^\infty (M)$, the \emph{bundle curvature}. Obviously, the bundle curvature is constant along fibers and, consequently, it can be seen as a function on the base surface $\tau\in\mathcal{C}^\infty(B)$ (denoted again by $\tau$).

From now on, the 3-manifold $M$ is going to be called the \emph{total space} of the Killing submersion and the surface $B$, the \emph{base surface}. It turns out that most of the geometry of a Killing submersion is encoded in a pair of functions on the base surface, namely, the \emph{Gaussian curvature} of the base surface $K_B$ and the \emph{bundle curvature} $\tau$. Therefore, it is usual to denote the total space of a Killing submersion by $M\equiv M(K_B,\tau)$.

A natural question that arises here concerns the existence of Killing submersions over a given surface $B$ (with Gaussian curvature $K_B$) for a prescribed bundle curvature $\tau\in\mathcal{C}^\infty(B)$. For arbitrary Riemannian surfaces $B$, existence has been proved in \cite{BGP}. In particular, if the surface $B$ happens to be simply connected then this result was already proven in \cite{M}. In the same paper, uniqueness (up to isomorphisms) is also guaranteed under the assumption that the total space $M$ is also simply connected.

Let $\gamma$ be an immersed curve in $B$, then the pre-image of $\gamma$ via the Killing submersion $\pi$, $S_\gamma:=\pi^{-1}(\gamma)$ is a surface isometrically immersed in $M$. Clearly, $S_\gamma$ is invariant under the one-parameter group of isometries associated with the vertical Killing vector field, $\xi$, $\mathcal{G}=\{\psi_t\,,\,t\in\mathbb{R}\}$. Indeed, any $\mathcal{G}$-invariant surface in $M$, $S$, is obtained by this construction for some curve $\gamma$ of $B$. These surfaces $S_\gamma$ are called \emph{vertical tubes} (or, \emph{vertical cylinders}) based on the curve $\gamma$.

Assume that $\gamma$ is parametrized by its arc-length $s$. Then, any horizontal lift of $\gamma$, $\bar{\gamma}$, is also arc-length parametrized. Now, using as coordinate curves the horizontal lifts of $\gamma$ and the fibers of the Killing submersion, vertical tubes $S_\gamma$ can be parametrized by
$$x(s,t)=\psi_t\left(\gamma(s)\right).$$
Note that $S_\gamma$ is embedded if $\gamma$ is a simple curve, and it is a torus when $\gamma$ is closed and $\mathcal{G}\cong \mathbb{S}^1$ is a circle group. Fibers can be considered to have finite length since, if necessary, a suitable quotient under a vertical translation can always be taken in order to get a circle bundle over $M$, hence, the second condition ($\mathcal{G}\cong\mathbb{S}^1$) can always be assumed. However, even in this case, horizontal lifts of $\gamma$, $\bar{\gamma}$, may not be closed due to the non-trivial holonomy (for an example, see \cite{ABG}).

Finally, as a consequence of the parametrization of these vertical tubes, we have that $S_\gamma$ are always flat ($K_{S_\gamma}\equiv 0$). Moreover, the mean curvature function of these surfaces, $H$, is closely related to the curvature function of the cross sections as the following formula shows (for details, see \cite{B})
\begin{equation}\label{k=2H}
H=\frac{1}{2}\left(\kappa\circ\pi\right),
\end{equation}
where $\kappa$ denotes the geodesic curvature of $\gamma$ in the base surface $B$.

Then, in this setting we can prove the following characterization for critical vertical tori.

\begin{theorem}\label{characterization} Let $\pi:M\rightarrow B$ be a Killing submersion with compact fibers and consider the mean curvature energy functional $$\mathcal{F}(S)=\int_S P\left(2 H\right)\,dA$$ acting on the space of surface immersions in the total space $M$. If $\gamma$ is a closed curve in $B$, then its vertical torus $S_\gamma=\pi^{-1}\left(\gamma\right)$ is a critical point of $\mathcal{F}(S)$ if and only if $\gamma$ is a critical curve of the following curvature energy $$\mathbf{\Theta}(\gamma)=\int_\gamma P\left(\kappa\right)\,ds\,.$$
\end{theorem}
\begin{proof} Since $\gamma$ is closed and the fibers of the Killing submersion are compact, $S_\gamma$ is a vertical torus based on $\gamma$. Now, we extend the action of $\mathcal{G}\cong\mathbb{S}^1$ on $M\equiv M(K_B,\tau)$ to $I(S,M)$ in a natural way,
$$\mathcal{F}(\phi)=\mathcal{F}(\psi_t\circ\phi)$$
for all $t\in\mathbb{R}$ and $\phi\in I(S,M)$.\\
We also identify the space of the $\mathbb{S}^1$-invariant immersions, $\Sigma$, with the space of vertical tori based on closed curves in $B$, i.e.
$$\Sigma\equiv\{S_\gamma=\pi^{-1}(\gamma)\,,\, \gamma \text{ is a closed curve in } M\}\,.$$
Then, we are in conditions to apply the Symmetric Criticality Principle of Palais, \cite{P}, to reduce symmetry. That is, we have that $S_\gamma$ is a critical point of $\mathcal{F}(S)$ acting on the space of surface immersions in the total space $M$ if and only if it is a critical point of $\mathcal{F}(S)$ restricted to $\Sigma$, $\mathcal{F}_{\lvert_\Sigma}(S)$.\\
Finally, using \eqref{k=2H} we conclude that this happens, precisely, when $\gamma$ is a critical curve of $\mathbf{\Theta}$. \end{proof}

To end this section, we will recall a few basic concepts about variational problems for curves. We will begin with the Euler-Lagrange equation of 
\begin{equation}\label{energyB}
\mathbf{\Theta}(\gamma)=\int_\gamma P\left(\kappa\right)\,ds
\end{equation}
in any surface $B$ with Gaussian curvature $K_B$. This equation is given by
\begin{equation}
\dot{P}_{ss}+\dot{P}\left(\kappa^2+K_B\right)-\kappa P=0\,, \label{ELB}
\end{equation}
where the upper dot denotes the derivative with respect to $\kappa$ and $s$ is the arc-length parameter along $\gamma$.

\begin{remark} Note that critical curves of $\mathbf{\Theta}(\gamma)=\int_\gamma P(\kappa)\,ds$ and $\widetilde{\mathbf{\Theta}}(\gamma)=\int_\gamma \mu P(\kappa)\,ds$ for any non-zero constant $\mu$ are the same, since the corresponding Euler-Lagrange equations coincide.
\end{remark}

Let $\gamma$ be any curve in a surface $B$ ($\gamma$ not necessarily closed). Choose any function $\tau\in\mathcal{C}^\infty(B)$ and construct the Killing submersion $\pi:M(K_B,\tau)\rightarrow B$ whose existence is guaranteed as explained above. Then, on the vertical tube based on $\gamma$, $S_\gamma$, the following formula holds, $$2R+\Ricci\left(\eta,\eta\right)=K_B\,,$$ where $R$ is the extrinsic Gaussian curvature of $M(K_B,\tau)$ and $\eta$ denotes the unit normal along $S_\gamma$. 

Hence, combining \eqref{k=2H} with $K_{S_\gamma}=0$ (since $S_\gamma$ is flat) we get that the Euler-Lagrange equation for $\mathbf{\Theta}(\gamma)$, \eqref{ELB}, becomes \eqref{EL} and, therefore, we conclude with the following result.

\begin{proposition} Let $\gamma$ be a curve in any surface $B$ and take any function $\tau\in\mathcal{C}^\infty(B)$. Then, the vertical tube based on $\gamma$, $S_\gamma$, verifies the Euler-Lagrange equation of $\mathcal{F}(S)$, \eqref{EL}, in $M(K_B,\tau)$ if and only if $\gamma$ verifies the Euler-Lagrange equation of $\mathbf{\Theta}(\gamma)$, \eqref{ELB}, in $B$.
\end{proposition}

In particular, if $\gamma$ happens to be closed in $B$ we can always consider a Killing submersion with compact fibers. Then, by a simple analysis of boundary terms, we recover Theorem \ref{characterization}.

\section{Examples in Homogeneous 3-Spaces}\label{Sec4}

Among Killing submersions, the most remarkable family are the homogenous 3-manifolds. As proved by Cartan, \cite{C2}, simply connected homogeneous Riemannian 3-manifolds with isometry group of dimension $6$ or $4$ can be represented by a 2-parameter family $\mathbb{E}(a,b)$, with $a$, $b\in\mathbb{R}$, but for the hyperbolic 3-space $\mathbb{H}^3(-1)$. These $\mathbb{E}(a,b)$ spaces determine Killing submersions over the simply connected constant Gaussian curvature surfaces $B\equiv B(\rho)$ where $K_B=\rho\in\mathbb{R}$. A local description of these $\mathbb{E}(a,b)$ spaces can be given by using the so-called \emph{Bianchi-Cartan-Vranceanu spaces} (BCV spaces, for short). 

The BCV spaces are described by the following 2-parameter family of Riemannian metrics
$$g_{a,b}=\frac{dx^2+dy^2}{\left[1+a\left(x^2+y^2\right)\right]^2}+\left(dz+b\,\frac{y dx-x dy}{2\left[1+a\left(x^2+y^2\right)\right]}\right)^2,\quad a,b\in\mathbb{R}$$
defined on $M=\{\left(x,y,z\right)\in\mathbb{R}\,,\, 1+a\left(x^2+y^2\right)>0\}$ (see \cite{C2} and \cite{V}). Regarded as Killing submersions, a simple computation shows that the Gaussian curvature of the base surface $B$ is $K_B=\rho=4a$ and that the bundle curvature is given by $\tau=b/2$. Thus, BCV spaces can be seen as the canonical models of Killing submersions with constant bundle and Gaussian curvatures (see, for instance, \cite{M}).

Among its simply connected members, these spaces include the \emph{Berger spheres}, the \emph{Riemannian products} $\mathbb{S}^2\times\mathbb{R}$ and $\mathbb{H}^2\times\mathbb{R}$, the \emph{Heisenberg group} $\mathbb{H}_3$, and the universal covering of the special linear group, $SL(2,\mathbb{R})$. However, this family also includes quotients of these spaces by suitable isometry subgroups, namely, 
\begin{itemize}
\item \emph{Lens spaces} $L_n=\mathbb{S}^3/\mathbb{Z}_n$ for $n\geq 2$ (including the real projective space $\mathbb{R}\mathbb{P}^3=L_2$) with the corresponding induced Berger metrics;
\item Heisenberg bundles (including those over flat tori); and
\item The projective special linear group $PSL(2,\mathbb{R})=\widetilde{PSL}(2,\mathbb{R})/\mathbb{Z}_2$ and other quotients of $\widetilde{PSL}(2,\mathbb{R})$.
\end{itemize}

From this and Theorem \ref{characterization}, it makes sense to study closed critical curves of $\mathbf{\Theta}(\gamma)$, \eqref{energyB}, in $B(\rho)$. In this setting, whenever $\dot{P}_s\neq 0$, the Euler-Lagrange equation, \eqref{ELB}, can be integrated once by multiplying by $\dot{P}_s$. Indeed, we have that
\begin{equation}
\dot{P}_s^2+\left(\kappa\dot{P}-P\right)^2+\rho\dot{P}^2=d \label{FIB}
\end{equation}
for a constant of integration $d$ represents a first integral of \eqref{ELB}. However, the parameter $d$ here is not completely free. In fact, if $\rho\geq 0$ then $d>0$. As will be clear in next section, throughout this paper we are just interested in the case $d>0$ for any value of $\rho$. Hence, for $d>0$, as we will see, most of the closed cases happen in the sphere, i.e. $B(\rho)=\mathbb{S}^2(\rho)$.

We now consider some interesting particular cases:

\subsection{Bending Energy}\label{benen}

According to the classical model of Euler-Bernoulli, the \emph{bending energy} functional of a curve is defined by $\mathbf{\Theta}(\gamma)$, \eqref{energyB}, for $P(\kappa)=\kappa^2+\lambda$ where $\lambda\in\mathbb{R}$. The constant $\lambda$ may be understood as a Lagrange multiplier constraining the length of the curves. Then, for any $\lambda$, critical curves of this energy are called \emph{(constrained) elastic curves}. 

However, one can consider no constraint on the length of the curves, which means $\lambda=0$. Thus, we consider the following \emph{bending energy}
\begin{equation}
\mathcal{E}(\gamma)=\int_\gamma \kappa^2\,ds\, .\label{BE}
\end{equation}
Critical curves of $\mathcal{E}(\gamma)$, \eqref{BE}, are called \emph{(free) elastic curves}. Observe that this bending energy is a particular case of $\mathbf{\Theta}(\gamma)$, \eqref{energyB}, corresponding with the choice $P(\kappa)=\kappa^2$. Free elastic curves in 2-space forms have been widely studied in the literature (see, for instance, \cite{LS}). In particular, the only closed critical curves appear in the sphere $\mathbb{S}^2(\rho)$. 

It turns out that there exists a beautiful link between elastica and Willmore surfaces. Indeed, this link was proved in \cite{Pi} using a symmetry reduction procedure as in Theorem \ref{characterization} for this particular case. Moreover, in the same paper, starting from closed constrained elastic curves in $\mathbb{S}^2(4)$, infinitely many closed Willmore surfaces in $\mathbb{S}^3(1)$ were found. Note that in this case, the Killing submersion is nothing but the standard Hopf mapping. 

\subsection{Extended Blaschke's Energy}\label{extblas}

In 1930 Blaschke studied the variational problem given by $\mathbf{\Theta}(\gamma)$, \eqref{energyB}, with $P(\kappa)=\sqrt{\kappa}$, but for variations of curves restricted to lie in the Euclidean 3-space, $\mathbb{R}^3$, \cite{Bl}. Much later, an extension of this functional was introduced in \cite{AGP2}, in order to study invariant constant mean curvature surfaces. This extension was done in two directions; first, both Riemannian and Lorentzian 3-space forms were considered as ambient spaces; and, second, an energy index was included in the curvature energy functional. 

Therefore, following \cite{AGP2}, the \emph{extended Blaschke's energy} is defined by
\begin{equation}
\mathcal{B}(\gamma)=\int_\gamma \sqrt{\kappa-\lambda}\,ds \label{EB}
\end{equation}
where $\lambda\in\mathbb{R}$ is a fixed constant. Note that this energy coincides with $\mathbf{\Theta}(\gamma)$, \eqref{energyB}, for $P(\kappa)=\sqrt{\kappa-\lambda}$.

Critical curves of $\mathcal{B}(\gamma)$, \eqref{EB}, in $B(\rho)$ were completely described in terms of their curvature in Corollary 3.3 of \cite{AGP2}. Here, we are mainly interested in closed critical curves with non-constant curvature. Therefore, since closed curves have periodic curvature, we restrict to this case and summarize the result in the following proposition.

\begin{proposition} Let $\gamma\subset B(\rho)$ be an extremal curve for $\mathcal{B}(\gamma)$, \eqref{EB}, with non-constant periodic curvature $\kappa(s)$. Then $\kappa(s)=\kappa_d(s)$ depends on a parameter $d$ and we have 
$$\kappa_d(s)=\frac{\rho+\lambda^2}{2d+\lambda-\sqrt{4d^2+4\lambda d-\rho}\sin\left(2\sqrt{\rho+\lambda^2}s\right)}+\lambda$$ 
for $d>\left(-\lambda+\sqrt{\rho+\lambda^2}\right)/2$ and $\rho+\lambda^2>0$.
\end{proposition}

In the case of $\mathbb{S}^2(\rho)$, existence of closed critical curves for any value of $\lambda$ was proved in \cite{AGP3}. In fact, those are the only closed curves for $d>0$.

Here, we would like to specialize the result of Theorem \ref{characterization} to critical curves of the extended Blaschke's energy. For this purpose we prove the existence of local variations such that all immersions in those variations verify $H>\mu$ for a fixed constant $\mu\in\mathbb{R}$, provided that the initial immersion verifies it.

\begin{proposition}\label{sub}
Let $\Phi:S\times\left(-\varepsilon,\varepsilon\right)\rightarrow M(\rho,\tau)$ be a variation of $\phi\in I(S,M(\rho,\tau))$ and assume that $H_\phi>\mu$ for a fixed constant $\mu\in\mathbb{R}$. Then there exists a positive $\hat{\varepsilon}\leq \varepsilon$ such that any immersion $\phi_\varsigma$ in the restriction of $\Phi$, $\hat{\Phi}:S\times\left(-\hat{\varepsilon},\hat{\varepsilon}\right)\rightarrow M(\rho,\tau)$, satisifies $H_{\phi_\varsigma}>\mu$.
\end{proposition}
\begin{proof} Consider a family of open sets in $\Phi\left(S\times\left(-\varepsilon,\varepsilon\right)\right)\subset M(\rho,\tau)$ such that they cover $\phi(S)$. Then, in any of these open sets, smoothness of $\Phi$ and $H$ implies that for a sufficiently small value $\widetilde{\varepsilon}\leq\varepsilon$, all the immersions $\phi_\varsigma$ satisfy $H_{\phi_\varsigma}>\mu$.\\
Moreover, since $S$ is compact, so is $\phi(S)$ and we can take a finite subfamily of open sets covering it. We call $\hat{\varepsilon}=\min \widetilde{\varepsilon}$ (where this minimum is only taken among the $\widetilde{\varepsilon}$ in this finite subfamily) and define $\hat{\Phi}$ as the restriction of $\Phi$, $\hat{\Phi}:S\times\left(-\hat{\varepsilon},\hat{\varepsilon}\right)\rightarrow M(\rho,\tau)$. Then, it is clear that all immersions in $\hat{\Phi}$ satisfy $H>\mu$.\\
Notice that both variations have the same variation vector field. \end{proof}

Then, the following result as a consequence of Theorem \ref{characterization} makes sense.

\begin{corollary}\label{ilus} Let $\gamma$ be a closed critical curve of the extended Blaschke's energy, $\mathcal{B}(\gamma)$, in $\mathbb{S}^2(\rho)$. Consider any function $\tau\in\mathcal{C}^\infty\left(\mathbb{S}^2(\rho)\right)$, then the vertical torus $S_\gamma=\pi^{-1}(\gamma)$ is a critical point of 
$$\mathcal{F}(S)=\int_S\sqrt{H-\mu}\,dA$$
for $\mu=\lambda/2$ in the total space of the Killing submersion $\pi:M(\rho,\tau)\rightarrow\mathbb{S}^2(\rho)$.
\end{corollary}

\subsection{Total Curvature Type Energies}\label{totalcurv}

Due to the classical results of Whitney and Grauestein, the total curvature functional acting on closed curves in the Euclidean plane is a trivial variational problem. Moreover, it can be checked that arbitrary planar curves are critical for the clamped problem.

However, in the literature, many generalizations of the total curvature have been studied. Here, we are going to consider the total curvature type energy used in \cite{Pa} to characterize rotational surfaces of constant Gaussian curvature in Riemannian 3-space forms. This energy is given by the functional
\begin{equation}
\mathcal{T}(\gamma)=\int_\gamma \sqrt{\varepsilon\left(\kappa^2+\lambda\right)}\,ds\,, \label{TC}
\end{equation}
where $\lambda$ is any real constant and $\varepsilon$ denotes the sign of $\kappa^2+\lambda$.

For the total curvature type energies $\mathcal{T}(\gamma)$, \eqref{TC}, in $B(\rho)$, the curvatures of the critical curves were obtained in Proposition 4.1 of \cite{Pa}. As before, in order to get closed curves we need periodic curvature. Next, we are going to summarize this result.

\begin{proposition} Let $\gamma\subset B(\rho)$ be an extremal curve for $\mathcal{T}(\gamma)$, \eqref{TC}, with non-constant periodic curvature $\kappa(s)=\kappa_d(s)$. Then $d\neq \varepsilon\lambda$, $\lambda<\rho$ and 
$$\kappa^2_d(s)=\frac{\lambda\left(\varepsilon d-\lambda\right)\sin^2\left(\sqrt{\rho-\lambda}s\right)}{\rho-\lambda-\left(\varepsilon d-\lambda\right)\sin^2\left(\sqrt{\rho-\lambda}s\right)}\,.$$
\end{proposition}

By a simple analysis of the closure conditions of critical curves of $\mathcal{T}(\gamma)$, \eqref{TC}, with periodic curvature, in Proposition 5.3 of \cite{Pa} it was proved that the only closed critical curves appear in the 2-sphere $\mathbb{S}^2(\rho)$. Now, specializing the result of Theorem \ref{characterization} to this base surface, we have,

\begin{corollary} Let $\gamma$ be a closed critical curve of the total curvature type energy $\mathcal{T}(\gamma)$ in $\mathbb{S}^2(\rho)$. Consider any function $\tau\in\mathcal{C}^\infty\left(\mathbb{S}^2(\rho)\right)$, then the vertical torus $S_\gamma=\pi^{-1}(\gamma)$ is a critical point of 
$$\mathcal{F}(S)=\int_S\sqrt{\varepsilon\left(H^2+\mu\right)}\,dA$$
for $\mu=\lambda/4$ in the total space of the Killing submersion $\pi:M(\rho,\tau)\rightarrow\mathbb{S}^2(\rho)$.
\end{corollary}

The particular case where $\mu>0$ was first studied in \cite{AGM}. In this case, $H^2+\mu$ is always positive, and there is no need of restriction on the variations. Observe that in the other cases the previous result makes sense, since an argument similar to that of Proposition \ref{sub} can be used to prove the existence of subvariations verifying the desired property.

\subsection{Astigmatism Energy}\label{astigm}

For any real constant $\lambda\in\mathbb{R}$, called the energy index, we consider here the curvature energy functional given by
\begin{equation}
\mathcal{A}(\gamma)=\int_\gamma \kappa\, e^{\lambda/\kappa}\,ds\,. \label{A}
\end{equation}
If $\lambda=0$, the curvature energy functional $\mathcal{A}(\gamma)$, \eqref{A}, is just the total curvature functional, hence, we assume here that $\lambda\neq 0$. 

This energy was introduced in \cite{LP2} in order to study rotational surfaces of constant astigmatism in Riemannian 3-space forms. Moreover, by studying some geometric properties of critical curves and the closure condition, in the same paper it was proved that the only closed critical curves with non-constant curvature appear in $\mathbb{S}^2(\rho)$.

Then, in this setting, we have the following result due to Theorem \ref{characterization}.

\begin{corollary} Let $\gamma$ be a closed critical curve of the total curvature type energy $\mathcal{A}(\gamma)$ in $\mathbb{S}^2(\rho)$. Consider any function $\tau\in\mathcal{C}^\infty\left(\mathbb{S}^2(\rho)\right)$, then the vertical torus $S_\gamma=\pi^{-1}(\gamma)$ is a critical point of 
$$\mathcal{F}(S)=\int_S H e^{\mu/H}\,dA$$
for $\mu=\lambda/2$ in the total space of the Killing submersion $\pi:M(\rho,\tau)\rightarrow\mathbb{S}^2(\rho)$.
\end{corollary}

\section{Binormal Evolution Tori}\label{Sec5}

Let us assume that $B(\rho)$ is isometrically immersed in a 3-dimensional Riemannian space form $M^3(\rho)$ as a totally geodesic surface. A planar curve (vanishing torsion) in $M^3(\rho)$ can be considered to lie in a totally geodesic surface, namely, $B(\rho)$. Conversely, a curve in $B(\rho)$ can be regarded as a planar curve in $M^3(\rho)$. Then, in this section we consider the evolution of a planar closed critical curve of $\mathbf{\Theta}(\gamma)$, \eqref{energyB}, under its associated binormal flow, \cite{AGP}.

We start by briefly recalling the notion of \emph{Killing vector fields along a curve}. A vector field $W$ along $\gamma$, which infinitesimally preserves unit speed parametrization is said to be a Killing vector field along $\gamma$ (in the sense of \cite{LS}) if $\gamma$ evolves in the direction of $W$ without changing shape, only position. It turns out that since $\gamma$ is critical for $\mathbf{\Theta}(\gamma)$, \eqref{energyB}, then the vector field defined on $\gamma$,
\begin{equation}
\mathcal{I}=\dot{P}(\kappa)B\label{I}
\end{equation}
where $B$ is the (constant) binormal vector field of $\gamma$, is a Killing vector field along it.

In space forms, Killing vector fields along curves can be uniquely extended to Killing vector fields defined in the whole space $M^3(\rho)$, \cite{LS}. Killing vector fields in $M^3(\rho)$ are the infinitesimal generators of isometries and, thus, evolution of curves under them generates invariant surfaces. In particular, the evolution under the extension of the Killing vector field along a curve in the direction of the binormal (binormal flow) generates an invariant surface of $M^3(\rho)$, usually called \emph{binormal evolution surface}, \cite{AGP}.

In our case, we are considering the binormal evolution surface swept out by evolving a critical curve $\gamma$ of $\mathbf{\Theta}(\gamma)$, \eqref{energyB}, under the extension of $\mathcal{I}$, \eqref{I}. It turns out that, since the initial curve of the evolution is planar, this binormal evolution surface is invariant under the action of a group of rotations, \cite{AGP}. 

\begin{remark} Note that if the initial curve has constant curvature, then the binormal evolution surface is just an isoparametric surface, i.e. in our context, it has constant principal curvatures. These surfaces in $M^3(\rho)$ were classified by Cartan in \cite{C}, proving that they are either totally umbilical or spherical cylinders. Hence, from now on, we are going to assume that the initial curve has non-constant curvature (or, equivalently, that the associated binormal evolution surface is non-isoparametric).
\end{remark}

Finally, for the non-constant curvature case, critical curves of $\mathbf{\Theta}(\gamma)$, \eqref{energyB}, are characterized by equation \eqref{FIB}. Following \cite{AGP}, we also have that the rotations associated with the binormal flow are of \emph{spherical type} (their orbits are Euclidean circles) if the constant of integration appearing in \eqref{FIB} is positive, i.e. $d>0$.

To sum up, let $\gamma$ be a critical curve of $\mathbf{\Theta}(\gamma)$, \eqref{energyB}, in $B(\rho)$ for $d>0$. We understand $\gamma$ as a planar curve in $M^3(\rho)$, as explained above. Then, we denote by $\widetilde{\mathcal{G}}=\{\widetilde{\psi}_t\,,\,t\in\mathbb{R}\}$ the group of (spherical) rotations associated with the extension of $\mathcal{I}$, \eqref{I}, and consider the binormal evolution surface of $M^3(\rho)$ (which is rotational) given by this binormal flow, $\widetilde{S}_\gamma:=\{\widetilde{\psi}_t\left(\gamma(s)\right)\}$. Therefore, if $\gamma$ is closed, then the binormal evolution surface $\widetilde{S}_\gamma$ generated as before is a torus in $M^3(\rho)$, called from now on \emph{binormal evolution torus}.

A nice geometric property of binormal evolution tori is the following.

\begin{theorem}\label{property} Let $\gamma$ be a closed critical curve of $\mathbf{\Theta}(\gamma)=\int_\gamma P(\kappa)ds$ in $B(\rho)$ for any $d>0$ and consider the binormal evolution torus $\widetilde{S}_\gamma$ generated by evolving $\gamma$ under the binormal flow associated to $\mathcal{I}=\dot{P}(\kappa)B$. Then $\widetilde{S}_\gamma$ verifies
\begin{equation}\label{WR}
\kappa_1=\kappa_2-\frac{P(\kappa)}{\dot{P}(\kappa)}
\end{equation}
between its principal curvatures, $\kappa_1$ and $\kappa_2$.
\end{theorem}
\begin{proof} We know that our closed critical curve $\gamma$ evolves by (spherical) rotations of $M^3(\rho)$, sweeping out a binormal evolution torus $\widetilde{S}_\gamma$, which can be naturally parametrized by 
$$y(s,t)=\widetilde{\psi}_t\left(\gamma(s)\right).$$
Since $\widetilde{\psi}_t$ are rotations, the velocity of the binormal evolution is given by
$$G(s):=\langle y_t,y_t\rangle^{1/2}=\langle\mathcal{I},\mathcal{I}\rangle^{1/2}=\dot{P}\left(\kappa(s)\right).$$
Then, with respect to the above natural parametrization of $\widetilde{S}_\gamma$, the induced metric can be written as
\begin{equation}
g=ds^2+\dot{P}^2\left(\kappa(s)\right) dt^2\,. \label{g}
\end{equation}
The Christoffel symbols of the Levi-Civita connection of \eqref{g} can be expressed in terms of the metric coefficients $g_{ij}$. Then, since $\gamma_t(s)$ are congruent copies of $\gamma(s)$, they are Frenet curves (recall that the case with constant curvature gives rise to isoparametric surfaces and that we are not considering it, although it trivially verifies the condition \eqref{WR}). Denoting by $\{T(s,t),N(s,t),B(s,t)\}$ to their corresponding Frenet frame, we can choose the following adapted frame on $\widetilde{S}_\gamma$,
$$e_1=y_s=T\,\quad e_2=\frac{y_t}{\dot{P}}=B\,,\quad e_3=\eta=T\times B=-N\,,$$
$e_1$ and $e_2$ being tangent to $\widetilde{S}_\gamma$ whilst $e_3=\eta$ is a unit normal on it.\\ Now, after long straightforward computations using the Gauss and Weingarten formulae (for details see \cite{AGP}), the principal curvatures of $\widetilde{S}_\gamma$ can be computed obtaining that $\kappa_1=-\kappa(s)$ and $\kappa_2=h_{22}(s)$, where $h_{22}(s)$ is the second coefficient of the second fundamental form and it is given by (see \cite{AGP})
$$h_{22}(s)=\frac{1}{\kappa(s)}\left(\frac{\dot{P}_{ss}}{\dot{P}}+\rho\right).$$
Direct computations using the Euler-Lagrange equation \eqref{ELB} and above expressions for the principal curvatures lead, after some simplifications, to \eqref{WR}, proving the result. \end{proof}

Equation \eqref{WR} implies that the binormal evolution tori $\widetilde{S}_\gamma$ are \emph{Weingarten tori}. A surface in a Riemannian 3-space form $M^3(\rho)$ is said to be a \emph{Weingarten surface} if the two principal curvatures, $\kappa_1$ and $\kappa_2$, satisfy a non-trivial functional relation $W(\kappa_1,\kappa_2)=0$ along the surface. These surfaces were introduced by Weingarten in \cite{We} and its study occupies an important role in Classical Differential Geometry, as the works of Chern, Hartman, Hopf and Wintner (among others) in the decade of 1950, prove.

One of the simplest Weingarten relations is the so-called \emph{linear Weingarten relation}, although it is an affine relation between the principal curvatures, i.e.
\begin{equation}
\kappa_1=a\kappa_2+b\label{LW}
\end{equation}
where $a$, $b\in\mathbb{R}$. Here, we are going to avoid the trivial examples for the case $a=0$.

Consider first that $a\neq 1$. Then, following \cite{LP}, observe that for a binormal evolution torus if we plug \eqref{LW} into \eqref{WR}, we end up with (recall that the principal curvature $\kappa_1$ is given in terms of the curvature of the profile curve $\gamma$, $\kappa$, by $\kappa_1=-\kappa$)
$$\frac{\dot{P}(\kappa)}{P(\kappa)}=\frac{a}{\left(1-a\right)\kappa_1-b}=\frac{a}{a-1}\cdot\frac{1}{\kappa-b/(a-1)}$$
since $a\neq 1$. This differential equation can be integrated by simple quadratures obtaining (up to a multiplicative constant)
\begin{equation}\label{qelas}
P(\kappa)=\left(\kappa-\lambda\right)^q
\end{equation}
where $\lambda=b/(a-1)$ and $q=a/(a-1)$. Critical curves of the energy $\mathbf{\Theta}(\gamma)$, \eqref{energyB}, for the above function $P(\kappa)$, \eqref{qelas}, have been studied in the literature and they are called \emph{q-elastic curves}. In particular, among this family we highlight the case where $a=-1$ and $b\in\mathbb{R}$. In this case, $\widetilde{S}_\gamma$ has constant mean curvature $H=b/2=-\lambda$ and $\mathbf{\Theta}(\gamma)$, \eqref{energyB}, is the extended Blaschke's energy \eqref{EB} (see Subsection \ref{extblas}).

Let's now take $a=1$. Then, $\kappa_1=\kappa_2+b$, and therefore, for a binormal evolution torus it is clear that $P(\kappa)=-b\dot{P}(\kappa)$ which can be integrated obtaining (up to a multiplicative constant)
$$P(\kappa)=e^{\lambda\kappa}$$
for $\lambda=-b$. Observe that this case also falls inside the energy $\mathbf{\Theta}(\gamma)$, \eqref{energyB}. These critical curves and their associated binormal evolution have been used in \cite{LP1} to classify rotational surfaces of constant skew curvature in $3$-space forms.

Moreover, we can also consider non-linear Weingarten relations $W(\kappa_1,\kappa_2)=0$. For instance, surfaces of constant Gaussian curvature $K_o$ in Riemannian 3-space forms $M^3(\rho)$ can be described in terms of their principal curvatures, using the Gauss equation, by
$$\kappa_1\kappa_2=K_o-\rho\,.$$
Substituting it in \eqref{WR} we get (once more, take into account that for binormal evolution tori $\kappa_1=-\kappa$)
$$\dot{P}(\kappa)\left(\kappa^2+\rho-K_o\right)=\kappa P(\kappa)\,.$$
Now, integrating, we conclude with (up to a multiplicative constant)
$$P(\kappa)=\sqrt{\varepsilon\left(\kappa^2+\lambda\right)}$$
where $\lambda=\rho-K_o$ and $\varepsilon$ is the sign of $\kappa^2+\lambda$. Therefore, $\mathbf{\Theta}(\gamma)$, \eqref{energyB}, is nothing but the total curvature type energy \eqref{TC} (see Subsection \ref{totalcurv}).

Finally, for the sake of completeness, we also recover the astigmatism energy \eqref{A}. A surface of constant astigmatism is a surface where the principal curvatures verify
$$\frac{1}{\kappa_1}-\frac{1}{\kappa_2}=c$$
along the surface. Hence, they can be understood as Weingarten surfaces for the relation $\kappa_1=c\kappa_1\kappa_2+\kappa_2$. Then, using \eqref{WR}, we get the following differential equation (again, $\kappa_1=-\kappa$)
$$\kappa^2\dot{P}(\kappa)=\left(\kappa-\lambda\right)P(\kappa)$$
where $\lambda=1/c$ (the case $c=0$ is avoided since it corresponds with totally umbilical surfaces). Up to a multiplicative constant, we can integrate the above differential equation to obtain
$$P(\kappa)=\kappa e^{\lambda/\kappa}\,.$$
Then, $\mathbf{\Theta}(\gamma)$, \eqref{energyB}, is the astigmatism energy \eqref{A} (see Subsection \ref{astigm}) which determines the profile curves of binormal evolution tori with constant astigmatism.

\section{Correspondence Results}\label{Sec6}

In this section, we relate both constructions of tori in order to show a correspondence between Weingarten tori in Riemannian 3-space forms and closed critical points of mean curvature energies in total spaces of Killing submersions.

We begin by proving the converse of Theorem \ref{property}.

\begin{theorem}\label{converse} Let $S$ be a non-isoparametric rotational Weingarten torus in $M^3(\rho)$ verifying a relation $W(\kappa_1,\kappa_2)=0$ between its principal curvatures. Then, the profile curve of $S$, $\gamma$, is locally a planar closed critical curve for
$$\mathbf{\Theta}(\gamma)=\int_\gamma P(\kappa)\,ds$$
where $P(\kappa)$ is given implicitly by $W(-\kappa,P(\kappa)/\dot{P}(\kappa)-\kappa)=0$.
\end{theorem}
\begin{proof} Any rotational torus $S$ of $M^3(\rho)$ is locally described as $\widetilde{S}_\gamma$, where $\gamma$ is a planar curve in $M^3(\rho)$, the profile curve. Therefore, we can assume that $\gamma$ lies fully in $B(\rho)$, a totally geodesic surface of $M^3(\rho)$.\\
From the natural parametrization of $\widetilde{S}_\gamma$, $y(s,t)=\widetilde{\psi}_t\left(\gamma(s)\right)$ where $\widetilde{\psi}_t$ is the rotation that leaves $\widetilde{S}_\gamma$ invariant, it is clear that $\kappa_1=-\kappa$ and 
$$\kappa_2=h_{22}=\frac{1}{\kappa}\left(\frac{G_{ss}}{G}+\rho\right)$$
$\kappa$ being the curvature of $\gamma$ in $B(\rho)$ and $G(s)=\langle y_t,y_t\rangle^{1/2}$ (for details see \cite{AGP}).\\
Now, after long, straightforward computations, one can see that the Gauss and Weingarten formulae lead to a PDE system to be satisfied by $y$ (see, for instance, \cite{AGP}). The compatibility conditions for this system are given by the Gauss-Codazzi equations, which in our case, since $\widetilde{\psi}_t$ are rotations and $\gamma$ is planar, boil down to
$$0=\left(\frac{1}{\kappa}\left[G_{ss}+G\left(\kappa^2+\rho\right)\right]\right)_s-\kappa_sG=\left(G h_{22}\right)_s+G_s \kappa\,.$$
Moreover, since the curvature of $\gamma$ is not constant (recall that $S$ is non-isoparametric), locally, by the Inverse Function Theorem, we can suppose that $s$ is a function of $\kappa$ and define $\dot{Q}(\kappa)=G(\kappa)$, where the upper dot denotes derivative with respect to $\kappa$. Therefore, the Gauss-Codazzi equation can be expressed in the following way
\begin{equation}\label{proof}
\dot{Q}_{ss}+\dot{Q}\left(\kappa^2+\rho\right)-\kappa Q=\lambda \kappa\,,
\end{equation}
for some $\lambda\in\mathbb{R}$. If we call $P(\kappa)=Q(\kappa)+\lambda$, then $\gamma$ is a planar closed critical curve of $\mathbf{\Theta}(\gamma)$, \eqref{energyB}.\\
Finally, note that the Weingarten relation becomes
$$0=W(\kappa_1,\kappa_2)=W(-\kappa,h_{22})=W(-\kappa,P/\dot{P}-\kappa)$$
since using \eqref{proof} together with $P(\kappa)=Q(\kappa)+\lambda$ in the definition of $h_{22}$ gives that
$$h_{22}=\frac{Q+\lambda}{\dot{Q}}-\kappa=\frac{P}{\dot{P}}-\kappa\,.$$
That is, the function $P(\kappa)$ is implicitly given by the Weingarten relation. \end{proof}

Observe that if $S$ is a non-isoparametric rotational torus in $M^3(\rho)$, then locally it is always a Weingarten torus. Indeed, as described in the previous proof, $S$ can be locally described as $\widetilde{S}_\gamma$ where $\gamma\subset B(\rho)$ (a totally geodesic surface of $M^3(\rho)$), and the principal curvatures are $-\kappa$ and $h_{22}$, respectively. In particular, from the definition of $h_{22}$ (or, equivalently, from the Gauss equation) we have
$$\kappa_2=h_{22}=-\frac{1}{\kappa_1}\left(\frac{G_{ss}}{G}+\rho\right)$$
where $G\equiv G(s)$. Finally, recall that from the Inverse Function Theorem, locally, the parameter $s$ can be regarded as a function of $\kappa=-\kappa_1$ and, hence, we have a relation between the principal curvatures, $\kappa_1$ and $\kappa_2$.

We prove now the main theorems of this section, which relate critical tori of mean curvature energies in total spaces of Killing submersions with Weingarten tori in Riemannian 3-space forms. The first result is the following.

\begin{theorem}\label{equiv1} Let $S$ be a rotational Weingarten torus in a Riemannian 3-space form $M^3(\rho)$ verifying $W(\kappa_1,\kappa_2)=0$ between its principal curvatures and let $B(\rho)$ denote the totally geodesic surface of $M^3(\rho)$ where the profile curve of $S$ lies. Then, for each function $\tau\in\mathcal{C}^\infty\left(B(\rho)\right)$, there exists a vertical torus critical for the mean curvature energy
$$\mathcal{F}(S)=\int_S P(2 H)\,dA$$
in the total space of the Killing submersion $\pi:M(\rho,\tau)\rightarrow B(\rho)$.
\end{theorem}
\begin{proof} From Theorem \ref{converse}, since $S$ is a rotational Weingarten torus, then its profile curve $\gamma$ is a planar closed critical curve of $\mathbf{\Theta}(\gamma)$, \eqref{energyB}, where $P(\kappa)$ is determined by the Weingarten relation. That is, planarity implies that $\gamma$ is fully contained in the totally geodesic surface $B(\rho)$. \\
Now, for a given function $\tau\in\mathcal{C}^\infty\left(B(\rho)\right)$ we construct the Killing submersion over $B(\rho)$ with bundle curvature $\tau$, $\pi:M(\rho,\tau)\rightarrow B(\rho)$. Recall that the existence of these Killing submersions was proved in \cite{BGP} (see also \cite{M}).\\
Finally, we apply Theorem \ref{characterization} to draw the conclusion. \end{proof}

Conversely, we have the following theorem.

\begin{theorem}\label{equiv2} Let $\pi:M\rightarrow B(\rho)$ be any Killing submersion and denote by $S_\gamma$ a vertical torus based on $\gamma\subset B(\rho)$ critical for the mean curvature energy
$$\mathcal{F}(S)=\int_S P(2 H)\,dA$$
in $M$. Then, the binormal evolution torus $\widetilde{S}_\gamma$ in $M^3(\rho)$ generated by evolving $\gamma$ under its associated binormal flow verifies the relation between its principal curvatures
$$\kappa_1=\kappa_2-\frac{P(\kappa)}{\dot{P}(\kappa)}$$
where $\kappa=-\kappa_1$ is the curvature of $\gamma$.
\end{theorem}
\begin{proof} For any Killing submersion, since $S_\gamma$ is a vertical torus, the curve $\gamma$ is necessarily closed. Moreover, as a consequence of Theorem \ref{characterization}, $S_\gamma$ critical for $\mathcal{F}(S)$, \eqref{energy}, implies that $\gamma\subset B(\rho)$ is a closed critical curve of $\mathbf{\Theta}(\gamma)$, \eqref{energyB}.\\
At this point, consider $B(\rho)\subset M^3(\rho)$ immersed as a totally geodesic surface, then the construction of binormal evolution tori explained in Section \ref{Sec5}, together with Theorem \ref{property}, concludes the proof. \end{proof}

In particular, let us assume that $\tau\in\mathbb{R}$ so that $M(\rho,\tau)$ becomes a homogeneous 3-space locally described by a BCV space, i.e. $\mathbb{E}(\rho/4,2\tau)$. Then, using Theorem \ref{equiv1} and Theorem \ref{equiv2}, we conclude with the following corollary.

\begin{corollary}\label{last} Let $M=\mathbb{E}(\rho/4, 2\tau)$ with $\tau\in\mathbb{R}$. Then the vertical torus $S_\gamma$ based on $\gamma\subset B(\rho)$ is critical for the mean curvature energy 
$$\mathcal{F}(S)=\int_S P(2H)\,dA$$
in $M$ if and only if the binormal evolution torus $\widetilde{S}_\gamma$ in $M^3(\rho)$ generated by evolving $\gamma$ under its associated binormal flow is a Weingarten torus for the relation
$$\kappa_1=\kappa_2-\frac{P(\kappa)}{\dot{P}(\kappa)}$$
between its principal curvatures, where $\kappa=-\kappa_1$ denotes the curvature of $\gamma$.\\ Moreover, for each of the binormal evolution tori there are either three $(\rho=0)$ or four $(\rho\neq 0)$ spaces where the associated vertical tori are critical, depending on the constant $\tau\in\mathbb{R}$.
\end{corollary}

As an illustration, we consider the 2-dimensional sphere of curvature $4$, $\mathbb{S}^2(4)$, and take a constant bundle curvature, $1=\tau\in\mathcal{C}^\infty\left(\mathbb{S}^2(4)\right)$. For these choices, the Killing submersion $\pi$ is just the \emph{standard Hopf mapping}
$$\pi:\mathbb{S}^3(1)\rightarrow \mathbb{S}^2(4)$$
which can be defined as follows. Take the complex plane $\mathbb{C}$ and endow $\mathbb{C}^2$ with the Riemannian metric
$$\widetilde{g}\left((z_1,z_2),(\omega_1,\omega_2)\right)=\Re \left(z_1\bar{\omega}_1+z_2\bar{\omega}_2\right)$$
where $z_i$ and $\omega_i$, $i=1,2$ are any complex numbers. In $\mathbb{C}^2$, define the map $\widetilde{\pi}:\mathbb{C}^2\rightarrow\mathbb{C}^2$ by
$$\widetilde{\pi}(z,\omega)=\frac{1}{2}\left(\lvert z\rvert^2-\lvert\omega\rvert^2,2\bar{z}\omega\right).$$
Then, the restriction of $\widetilde{\pi}$ to the hyperquadric $\widetilde{g}\left((z_1,z_2),(z_1,z_2)\right)=1$, i.e. to the 3-dimensional sphere $\mathbb{S}^3(1)$ gives the standard Hopf mapping $\pi$.

Let $\gamma$ be a curve in $\mathbb{S}^2(4)$. Then, as mentioned before, its vertical lift $S_\gamma$ is a flat surface in $\mathbb{S}^3(1)$. The covering map $\Psi:\mathbb{R}^2\rightarrow S_\gamma$ defined by $\Psi(s,t)=e^{it}\bar{\gamma}(s)$, where $\bar{\gamma}(s)$ denotes a horizontal lift of $\gamma$ can be used to parametrize $S_\gamma$. Indeed, explicit parametrizations of $S_\gamma$ can be obtained in the following way: take an arbitrary curve $\gamma(s)=\left(A_1(s),0,A_2(s),A_3(s)\right)$ in $\mathbb{S}^2(4)$, then the horizontal lifts of $\gamma$ via $\pi$ are given by
\begin{eqnarray*}
\bar{\gamma}(s)&=&\left(\alpha_1(s)\cos\beta(s),\alpha_1(s)\sin\beta(s),\alpha_2(s)\cos\beta(s)-\alpha_3(s)\sin\beta(s),\right.\\
&&\left. \, \alpha_2(s)\sin\beta(s)+\alpha_3(s)\cos\beta(s)\,\right),
\end{eqnarray*}
where
$$\alpha_1(s)=\sqrt{A_1(s)+1/2}\,,\quad\quad\quad \alpha_i(s)=\frac{A_i(s)}{\sqrt{A_1(s)+1/2}}$$
for $i=2,3$; and,
$$\beta(s)=\pm\int \frac{A_3(s)A_2'(s)-A_2(s)A_3'(s)}{A_1(s)+1/2}\,ds\,.$$
Recall that if $\gamma$ is a closed curve, then $S_\gamma$ is a closed surface; i.e. a flat torus. However, the horizontal lift of $\gamma$, $\bar{\gamma}$, may not be closed because the non-trivial holonomy. Although, if in addition, the area enclosed by $\gamma(s)$ in $\mathbb{S}^2(4)$ is a rational multiple of $\pi$, then there are $m\in\mathbb{Z}$ such that the horizontal lift of an $m$-cover of $\gamma(s)$ is closed, \cite{ABG}.

Let us consider now the Blaschke's energy $\mathbf{\Theta}(\gamma)$, \eqref{energyB}, for $P(\kappa)=\sqrt{\kappa}$, that is, $\mathcal{B}(\gamma)$, \eqref{EB}, with $\lambda=0$, acting on the space of curves immersed in $\mathbb{S}^2(4)$. As proved in \cite{AGP3}, there exists a biparametric family of closed critical curves $\{\gamma_{m,n}\}$ for $m<2n<\sqrt{2}m$. The parameters $n$ and $m$ have a geometric meaning. Indeed, $n$ denotes the number of times the critical curve goes around the pole to close up, and $m$ is the number of lobes the critical curve has.

These critical curves were used to construct (immersed) minimal rotational tori in $\mathbb{S}^3(4)$, \cite{AGP3} (see also Section \ref{Sec5}). On the other hand, if we combine the above information with Corollary \ref{ilus}, we conclude with the existence of a biparametric family of Hopf tori critical for the energy $\mathcal{F}(S)$, \eqref{energy} for $P(2H)=\sqrt{H}$. In fact, Corollary \ref{last} relates both tori.

In Figure \ref{dib} we show this relation. For the values $n=2$ and $m=3$, we consider the closed critical curve $\gamma_{3,2}$ for the Blaschke's energy $\mathbf{\Theta}(\gamma)$, \eqref{energyB}, for $P(\kappa)=\sqrt{\kappa}$ in $\mathbb{S}^2(4)$ (Figure \ref{dib}, left). We evolve this curve under its associated binormal flow to construct a minimal rotational torus in $\mathbb{S}^3(4)$ whose profile curve is, precisely, $\gamma_{3,2}$ (Figure \ref{dib}, center). Finally, we also consider the associated Hopf tori in $\mathbb{S}^3(1)$ based on $\gamma_{3,2}$ which is a critical flat torus for the energy $\mathcal{F}(S)$, \eqref{energy} for $P(2H)=\sqrt{H}$ (Figure \ref{dib}, right).

\begin{figure}[h!]
\makebox[\textwidth][c]{\centering\includegraphics{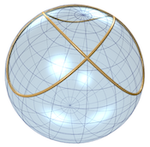}\quad\includegraphics{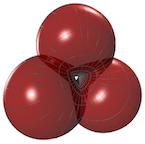}\quad\includegraphics{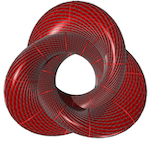}}
\caption{The closed critical curve $\gamma_{3,2}$ for the Blaschke's energy lying on the sphere $\mathbb{S}^2(4)$ (Left); the stereographic projection of the minimal rotational torus of $\mathbb{S}^3(4)$ constructed as the binormal evolution of $\gamma_{3,2}$ (Center); and, the stereographic projection of the associated Hopf torus based on $\gamma_{3,2}$ (Right).}
\label{dib}
\end{figure}

\section*{Acknowledgments}

Research partially supported by MINECO-FEDER grant PGC2018-098409-B-100, Gobierno Vasco grant IT1094-16 and by Programa Posdoctoral del Gobierno Vasco, 2018.

\end{document}